\documentclass{amsart}
\usepackage[english]{babel}
\usepackage[latin1]{inputenc}
 \usepackage[all]{xy}
 \usepackage[pagebackref]{hyperref}

\usepackage{amsmath,amsfonts,amssymb,amsthm,amscd,array,stmaryrd,mathrsfs, mathdots}
\usepackage[makeroom]{cancel}
\PassOptionsToPackage{table}{xcolor}
\usepackage{pstricks}
\usepackage{tikz}
\usepackage{graphicx}

\theoremstyle{plain}

\newtheorem{thm}{Theorem}[section]
\newtheorem{lem}{Lemma}[section]
\newtheorem{cor}[lem]{Corollary}
\newtheorem{prop}[lem]{Proposition}

\theoremstyle{definition}

\newtheorem{rem}[lem]{Remark}
\newtheorem{ex}[lem]{Example}

\let\ssection=\section
\renewcommand{\section}{\setcounter{equation}{0}\ssection}

\newcommand{\R}{\mathbb{R}}
\newcommand{\Z}{\mathbb{Z}}

\newcommand{\N}{\mathbb{N}}

\newcommand{\Q}{\mathbb{Q}}

\newcommand{\B}{\mathcal{B}}

\newcommand{\E}{\mathcal{E}}

\newcommand{\Tr}{\textup{Tr\,}}

\newcommand{\Id}{\mathrm{Id}}
\newcommand{\GL}{\mathrm{GL}}

\newcommand{\PGL}{\mathrm{PGL}}
\newcommand{\SL}{\mathrm{SL}}
\newcommand{\PSL}{\mathrm{PSL}}

\def\D{\Delta}

\def\s{\sigma}

\begin{document}

\title[$q$-deformed cubic equations]
{On $q$-deformed cubic equations:\\
the quantum heptagon and nonagon}

\author[V. Ovsienko]{Valentin Ovsienko}
\author[A. Ustinov]{Alexey Ustinov}

\address{
Valentin Ovsienko,
Centre National de la Recherche Scientifique,
Laboratoire de Math\'ematiques,
Universit\'e de Reims,
U.F.R. Sciences Exactes et Naturelles,
Moulin de la Housse - BP 1039,
51687 Reims cedex 2,
France}

\address{
Alexey Ustinov,
Faculty of Computer Science, HSE University, Moscow, Russia; and
Khabarovsk Division of the Institute for Applied Mathematics,
Far-Eastern Branch, Russian Academy of Sciences, Russia
}

\email{
valentin.ovsienko@univ-reims.fr,
Ustinov.Alexey@gmail.com}

\begin{abstract}

The recent notion of $q$-deformed irrational numbers is characterized by the invariance 
with respect to the action of the modular group $\PSL(2,\Z)$, 
or equivalently under the Burau representation of the braid group~$B_3$.
The theory of $q$-deformed quadratic irrationals and quadratic equations with integer coefficients
is known and entirely based on this invariance.
In this paper, we consider the case of cubic irrationals.
We show that irreducible cubic equations with three distinct real roots 
and cyclic Galois group~$C_3$ (or $\Z/3\Z$) acting by a third order element of $\PSL(2,\Z)$,
have a canonical $q$-deformation, that we describe.
This class of cubic equations contains well-known examples including the equations
that describe regular $7$- and $9$-gons.

\end{abstract}

\maketitle

\thispagestyle{empty}

\section{Introduction and main results}

The notion of $q$-deformed, or ``quantized'' rational numbers was introduced in~\cite{MGOfmsigma}
and extended to irrational numbers in~\cite{MGOexp}.
This $q$-deformation can be understood as a map from $\R_{\geq0}$ to the space of formal power series
in one variable, $q$, with integer coefficients
$$
\left[\,.\,\right]_q:\R_{\geq0}\to\Z[[q]]
$$
which is uniquely determined by the property of $\PSL(2,\Z)$-invariance
and the choice of the image of one point, namely $[0]_q:=0$.
For a real $x\geq0$, the $q$-deformation $\left[x\right]_q$
is a formal power series in~$q$.
This notion is extended to negative numbers using the formula
$$
\left[x-1\right]_q=\frac{[x]_q-1}{q},
$$
and the $q$-deformation $\left[x\right]_q$ for $x<0$ is defined as a Laurent series in~$q$.

Properties of $q$-real numbers have been studies by several authors.
In particular, the case of quadratic irrationals was studied in~\cite{LMGadv}, where it was proved
that, if $x_1$ and $x_2$ are roots of a quadratic equation
$$
ax^2+bx+c=0
$$
with integer coefficients $a,b,c$, then both infinite series
$\left[x_1\right]_q$ and $\left[x_2\right]_q$
satisfy an equation of the form
$$
A(q)X^2+B(q)X+C(q)=0,
$$
where $A(q),B(q),C(q)$ are some polynomials in~$q$,
and where the indeterminate $X$ is a Laurent power series in~$q$.

The problem to investigate the case of algebraic irrationals of higher degree
was already formulated in~\cite{MGOexp}.
We study this problem in the case where $x$ is a cubic irrational, that is a solution to
a cubic equation
\begin{equation}
\label{Cubic}
ax^3+bx^2+cx+d=0
\end{equation}
with $a,b,c,d\in\Z$.

We will assume that~\eqref{Cubic} is irreducible 
and has three distinct real roots, $x_1,x_2,x_3$.
This case is classically known under the name of {\it casus irreducibilis}.
Let
$$
X_1(q):=\left[x_1\right]_q,
\qquad
X_2(q):=\left[x_2\right]_q,
\qquad
X_3(q):=\left[x_3\right]_q
$$
be $q$-deformations of $x_1,x_2,x_3$ in the sense of~\cite{MGOexp}.
The precise problem that interests us is the following:
is there any $q$-analogue of the equation~\eqref{Cubic} 
to which the three series $X_1(q),X_2(q),X_3(q)$ are solutions?

We observe that, in general, the situation is much more complicated
than that of the quadratic case.
In particular, the above problem is out of reach so far.
However, we found two classes of cubic equations
that admit canonical $q$-deformations.

We consider two families of cubic equations
\begin{equation}
\label{Cubic1}
x^3 - bx^2 + \left(b-3\right)x + 1 = 0,
\end{equation}
and
\begin{equation}
\label{Cubic2}
x^3 - bx^2 - \left(b+3\right)x - 1 = 0,
\end{equation}
where $b\in\Z$ is a parameter.
For every~$b\in\Z$, the equation~\eqref{Cubic1} has two positive and one negative roots,
and the equation~\eqref{Cubic1} has one positive and two negative roots.

Both families are characterized by the fact that the discriminant is an integer square,
more precisely
$$
\D_-=(b^2 - 3b + 9)^2
\qquad\hbox{and}\qquad
\D_+=(b^2 + 3b + 9)^2,
$$ 
respectively,
and the Galois group, 
which has to be the cyclic group $\Z_3$ (or $C_3$, or $\Z/3\Z$),
acts on the roots as a third order element of $\PSL(2,\Z)$,
preserving the equation.

Note that the first family, at $b=0$, contains a remarkable equation
\begin{equation}
\label{NonaEq}
x^3 -3x + 1 = 0
\end{equation}
related to the regular nonagon.
Indeed, after substitution $x=z+z^{-1}$, the equation~\eqref{NonaEq} becomes
$\Phi_9(z)=0$, where $\Phi_9(z)=z^6+z^3+1$ is the $9$-th cyclotomic polynomial. 
Therefore, the roots of the equation~\eqref{NonaEq} are
\begin{equation}
\label{NonaRoots}
x_1=2\cos\left(\frac{2\pi}{9}\right),
\qquad
x_2=2\cos\left(\frac{4\pi}{9}\right),
\qquad
x_3=2\cos\left(\frac{8\pi}{9}\right).
\end{equation}

The second family, at $b=-1$, contains another remarkable equation
\begin{equation}
\label{HeptaEq}
x^3 + x^2 -2x - 1 = 0
\end{equation}
related to the regular heptagon. After substitution $x=z+z^{-1}$, the equation~\eqref{HeptaEq} 
becomes $\Phi_7(z)=0$, where $\Phi_7(z)=z^6+z^5+z^4+z^3+z^2+z+1$ is the $7$-th cyclotomic polynomial. 
In particular, the roots of~\eqref{HeptaEq} are
\begin{equation}
\label{HeptaRoots}
x_1=2\cos\left(\frac{2\pi}{7 }\right),
\qquad 
x_2=2\cos\left(\frac{4\pi}{7 }\right),
\qquad 
x_3=2\cos\left(\frac{8\pi}{7 }\right).
\end{equation}

Our main result is the following.

\begin{thm}
\label{FirstThm}
The $q$-deformations $X_1(q),X_2(q),X_3(q)$ of the roots of~\eqref{Cubic1},
resp. of the roots of~\eqref{Cubic2} 
satisfy the equation
\begin{eqnarray}
\label{Cubic1Quant}
X^3 - B_-(q)X^2 + \left(B_-(q)-3\right)X + 1 &=& 0,\\[6pt]
\label{Cubic2Quant}
X^3 - B_+(q)X^2 - \left(q^{-1}B_+(q)+3q^{-2}\right)X - q^{-3} &=& 0.
\end{eqnarray}
The coefficients~$B_-(q)$ and~$B_+(q)$  in~\eqref{Cubic1Quant} and~\eqref{Cubic2Quant} 
are a certain Laurent series in~$q$.
\end{thm}

The coefficients $B_\mp(q)$ in~\eqref{Cubic1Quant} and~\eqref{Cubic2Quant} 
is equal to the sum of the $q$-deformed roots
$$
B_\mp(q) = X_1(q)+X_2(q)+X_3(q),
$$
but there are other interesting formulas for calculating it.
The coefficient $B_\mp(q)$ can be viewed as a ``modular invariant'' of the corresponding equation,
since the definition of $q$-irrationals 
is based on the invariance with respect to the action of the modular group $\PSL(2,\Z)$.
Note that, even in the simplest cases,~\eqref{NonaEq} and~\eqref{HeptaEq}, 
the coefficients~$B_-(q)$ and~$B_+(q)$
seem to be an {\it infinite Laurent series},
although we are unable to prove this.

We formulate the following conjecture.
\textit{No $q$-deformed cubic irrational can satisfy a cubic equation 
of order~$3$ with polynomial coefficients.}
This conjecture is based on the analysis of the coefficients~$B_-(q)$ and~$B_+(q)$
of the equantions~\eqref{NonaEq} and~\eqref{HeptaEq} and many other
computer experiments with cubic equations 
from the families~\eqref{Cubic1} and ~\eqref{Cubic2}, but not only.

The paper is organized as follows.
In Section~\ref{DefSec} we recall the definition of $q$-deformed rationals and irrationals.
We pay much attention to the action of the modular group $\PSL(2,\Z)$,
which is the main ingredient of our proof of Theorem~\ref{FirstThm}.
In Section~\ref{ProofSec} we prove our main results,
and give more formulas for computing the coefficients~$B_\mp(q)$.
In Section~\ref{HeptaSec} we consider the
cubic equations~\eqref{NonaEq} and~\eqref{HeptaEq}
and give some details of computer calculations in these cases. 

\section{$q$-rationals and $q$-irrationals, a compendium}\label{DefSec}

In this section we review the definition of 
$q$-rationals and $q$-irrationals and the $q$-deformed action of the modular group $\PSL(2,\Z)$.

\subsection{The modular group action and its $q$-deformation}
The modular group $\PSL(2,\Z)$ 
acts on the rational projective line $\mathbb{P}(Q)=\Q\cup\{\infty\}$ by linear-fractional transformations: 
$$
M\left(x\right)=\frac{ax+b}{cx+d},
$$
where $x\in\mathbb{P}(Q)$ and
$$
M=\begin{pmatrix}
a&b\\
c&d\end{pmatrix},
\qquad\qquad 
a,b,c,d\in\Z,
\quad
ad-bc=1.
$$
This action is homogeneous.
The group $\PSL(2,\Z)$ can be generated by two elements
$$
T=\begin{pmatrix}
1&1\\[4pt]
0&1
\end{pmatrix},
\quad
\quad
S=\begin{pmatrix}
0&-1\\[4pt]
1&0
\end{pmatrix}
$$
with the relations $S^2=(TS)^3=\Id$.
The corresponding linear-fractional transformations are
$T\left(x\right)=x+1$ and $S\left(x\right)=-1/x$.

Following~\cite{MGOfmsigma,LMGadv},
consider the $\PSL(2,\Z)$-action on the space
$\Z(q)$ of rational functions in one variable~$q$ with integer coefficients is generated by
\begin{equation}
\label{TqSq}
T_q=\begin{pmatrix}
q&1\\[4pt]
0&1
\end{pmatrix},
\quad
\quad
S_q=\begin{pmatrix}
0&-1\\[4pt]
q&0
\end{pmatrix}.
\end{equation}
Note that $T_q$ and $S_q$, viewed as linear-fractional transformations on~$\Z(q)$
$$
T_q\left(F(q)\right)=qF(q)+1,
\qquad\qquad
S_q\left(F(q)\right)=-\frac{1}{qF(q)},
$$
satisfy the same relations as $T$ and $S$,
namely 
\begin{equation}
\label{RelEq}
S_q^2=(T_qS_q)^3=\Id.
\end{equation}

For every element $M\in\PSL(2,\Z)$, its $q$-deformation,
which is an element $M_q$ of $\PGL(2,\Z[q,q^{-1}])$, is now well-defined; see~\cite{LMGadv}.
Indeed, if $M$ is a monomial in $T$ and $S$, then by definition
$M_q$ is the same monomial in $T_q$ and $S_q$.
The definition is correct, i.e independent of the choice of the monomial, since
$T_q$ and $S_q$ satisfy the same relations~\eqref{RelEq} as $T$ and $S$.

\subsection{$q$-rationals}
The $\PSL(2,\Z)$-invariant quantization of rational numbers is the unique map
$$
\left[\,.\,\right]_q:\Q\cup\{\infty\}\to\Z(q)
$$
commuting with the $\PSL(2,\Z)$-action and sending $0$ to~$0$.
The uniqueness follows from the transitivity of the $\PSL(2,\Z)$-action on $\mathbb{P}(Q)$,
and the existence is guaranteed by the explicit formula for $q$-rationals 
in terms of continued fractions; see~\cite{MGOfmsigma} (and also~\cite{Sik}).

The $\PSL(2,\Z)$-invariance can be expressed as two recurrent formulas
\begin{equation}
\label{RecEq}
\left[x+1\right]_q=q\left[x\right]_q+1,
\qquad\qquad
\left[-\frac{1}{x}\right]_q=-\frac{1}{q\left[x\right]_q},
\end{equation}
that imply the following properties.

\begin{itemize}
\item
If $n\in\N$, then the first formula~\eqref{RecEq} implies that 
$[n]_q$ is the cyclotomic polynomial
\begin{equation}
\label{QuantInt}
[n]_q=1+q+\cdots+q^{n-1}=\frac{1-q^n}{1-q},
\end{equation}
which is commonly  used in many areas of algebra, combinatorics and mathematical physics
as the $q$-analogue of a natural number.
Note that~\eqref{QuantInt} explains why the $q$-deformed generator $T_q$
is chosen as in~\eqref{TqSq}.
The choice of $S_q$ is then determined by the relations~\eqref{RelEq}.

\medskip
\item
Let~$x\geq0$ be a rational number,
written as a standard finite continued fraction expansion that,
without loss of generality, can be chosen of even length
$x=[a_1;a_2,a_3,\ldots,a_{2m}]$.
Explicitly
$$
x
\quad=\quad
a_1 + \cfrac{1}{a_2 
          + \cfrac{1}{\ddots +\cfrac{1}{a_{2m}} } },
          $$
where $a_i\geq1$ (except for $a_1\geq0$).
Then the $q$-deformation of~$x$ is the following rational function in~$q$
\begin{equation}
\label{QuantRat}
[x]_{q}\quad=\quad
[a_1]_{q} + \cfrac{q^{a_{1}}}{[a_2]_{q^{-1}} 
          + \cfrac{q^{-a_{2}}}{[a_{3}]_{q} 
          +\cfrac{q^{a_{3}}}{[a_{4}]_{q^{-1}}
          + \cfrac{q^{-a_{4}}}{
        \cfrac{\ddots}{[a_{2m-1}]_q+\cfrac{q^{a_{2m-1}}}{[a_{2m}]_{q^{-1}}}}}.
          } }} 
\end{equation}
Notice that the odd coefficients $a_{2i+1}$ appear in~\eqref{QuantRat} 
in the accordance with~\eqref{QuantInt}, 
the even coefficients $a_{2i}$ are also $q$-deformed accordingly~\eqref{QuantInt},
but the parameter of the deformation is inverted.
\end{itemize}

For~$\frac{n}{m}\in\Q$, the $q$-deformation 
is a rational function in~$q$ with integer coefficients
$$
\left[\frac{n}{m}\right]=\frac{N(q)}{M(q)},
$$
where $N(q)$ and $M(q)$ are coprime monic polynomials with positive integer
coefficients.
Let us mention that $q$-deformed rational numbers enjoy a number of
remarkable properties, such as unimodality of the sequences of coefficients of
the polynomials $N(q)$ and $M(q)$~\cite{McCSS,OgRa},
total positivity and connection to the Jones polynomial~\cite{MGOfmsigma,KW,Sik},
relation to discrete geometry~\cite{Oven},
Farey tessellation~\cite{LMGCom},
Bernoulli numbers and Thomae's function~\cite{Las}, and many others.

\subsection{$q$-irrationals}

The notion of $q$-deformed irrational numbers is based on the following ``stabilization phenomenon''.

\begin{thm}[\cite{MGOexp}]
\label{StabThm}
(i)
Let $x\in\R\setminus\Q$, and let $(x_n)_{n\in\N}$
be a sequence of rationals converging to~$x$.
Then the Taylor series of the rational functions $[x_n]_q$ centered at $q=0$
stabilize (i.e converge in the topology of formal power series) to some power series,~$[x]_q$.

(ii)
The series $[x]_q$ does not depend on the choice
of the sequence $(x_n)_{n\in\N}$, but only on~$x$.
\end{thm}

Some properties of the series arising as $q$-deformed irrational numbers,
in particular their convergence radius, were studied in~\cite{LMGOV,RenB}.

\begin{rem}
If $x$ itself is a rational number, then the stabilization phenomenon has an amazing variant.
When $x$ is approximated by a sequence of rationals $(x_n)_{n\in\N}$ from the right,
the sequence of the rational functions $[x_n]_q$ converges to~$[x]_q$.
However, when $x$ is approximated by a sequence of rationals $(x_n)_{n\in\N}$ from the left,
the sequence of the rational functions $[x_n]_q$ converges to another rational function.
This was noticed in~\cite{MGOexp} and developed in~\cite{BBL},
leading to an intereesting notion of ``left $q$-rational''.
\end{rem}

In terms of the infinite  continued fraction expansion
$x=[a_1;a_2,a_3,a_4,\ldots]$, a useful explicit formula is
\begin{equation}
\label{QuantIrrat}
[x]_{q}\quad=\quad
[a_1]_{q} + \cfrac{q^{a_{1}}}{[a_2]_{q^{-1}} 
          + \cfrac{q^{-a_{2}}}{[a_{3}]_{q} 
          +\cfrac{q^{a_{3}}}{[a_{4}]_{q^{-1}}
          + \cfrac{q^{-a_{4}}}{
    \ddots}
          } }} 
\end{equation}
The stabilization phenomenon guarantees that cutting $n$ terms of
the infinite continued fraction~\eqref{QuantIrrat}
and developing the resulting rational function as a Taylor series,
more and more terms remain unchanged,
as~$n$ grows.

\subsection{Quadratic irrationals}
The $\PSL(2,\Z)$-invariance makes the situation of quadratic irrationals very simple; 
see~\cite{LMGadv}.

As commonly known, every quadratic irrational $x$ 
is a fixed point of some (hyperbolic) element $M$ of $\PSL(2,\Z)$,
that is, $M\left(x\right)=x$.
It follows from the $\PSL(2,\Z)$-invariance that
$$
M_q\big(\left[x\right]_q\big)= \left[x\right]_q.
$$
This implies that for $x=\frac{q+\sqrt{r}}{s}$ one has
$$
\left[x\right]_q=
\frac{Q(q)+\sqrt{R(q)}}{S(q)},
$$
where $Q(q),R(q),S(q)$ are some polynomials.
It was proved in~\cite{LMGadv}
that the polynomial $R(q)$ is palindromic (or ``reversal''), that is
$R(q^{-1})\deg(R)=R(q)$.
This follows from the palindromicity of the trace $\Tr(M_q)$, 
which is one of the main results of~\cite{LMGadv}.

Various combinatorial properties of $q$-deformed quadratic irrationals,
in particular relations to the Somos and Gale-Robinson sequences,
were studied in~\cite{OP}.

\subsection{Relation to the Burau representation}

An alternative understanding of the modular invariance, that provides a better understanding of 
the notion $q$-rationals, is based on the classical Burau representation of the braid group $\B_{3}$.

Let $\B_{3}$ be the braid group on three strands
(for the theory of braid groups and their applications; see~\cite{Bir}).
Let $\sigma_1,\sigma_2$ be the standard generators of~$\B_{3}$
satisfying the braid relation $\sigma_1\sigma_2\sigma_1=\sigma_2\sigma_1\sigma_2$.
The (reduced) Burau representation $\rho_3: \B_3 \to \GL(2, \Z[q,q^{-1}])$ is defined by
\begin{equation}
\label{rho3}
\begin{array}{rclrcl}
\rho_3\; : \quad \sigma_1 &\mapsto& 
\begin{pmatrix}
q&1\\[4pt]
0&1
\end{pmatrix}, & \sigma_2 &\mapsto& \begin{pmatrix}
1&0\\[4pt]
-q&q
\end{pmatrix},
\end{array}
\end{equation}
it is easy to check the braid relation
$$
\rho_3(\sigma_1)\rho_3(\sigma_2)\rho_3(\sigma_1)=
\rho_3(\sigma_2)\rho_3(\sigma_1)\rho_3(\sigma_2).
$$
Note that the centre $Z(\B_3) \subset \B_3$ acts trivially on $\Z(q)$ since 
$\rho_3((\s_1\s_2)^3)=-q^3 \Id$ is a scalar matrix, so that the effective
``projectivized Burau action'' is that of 
$$
\B_3/Z(\B_3)\cong\PSL(2,\Z).
$$
Furthermore, this action coincides with the action defined by~\eqref{TqSq}
since $\rho_3(\sigma_1)=T_q$ and $\rho_3(\sigma_2)=S_qT_qS_q$.

This connection between $q$-rationals and the Burau representation
proved to be useful in~\cite{MGOV}.
This viewpoint was extended in~\cite{Perrine} for the action of $\B_4$
on the projective plane $\mathbb{P}^2(\Q)$.

\section{Proof of Theorem~\ref{FirstThm}}\label{ProofSec}

In this section we prove our main result,
and reformulate it in terms of Vieta's formulas.
The main ingredient of the proof is $\PSL(2,\Z)$-invariance.
We also give another way to calculate the coefficient~$B(q)$.

\subsection{The Galois group action}

The Galois group of~\eqref{Cubic1} and~\eqref{Cubic2} is
isomorphic to the cyclic group of order~$3$.
This is equivalent to the fact that the discriminant of both equations
is a rational square.
However, these equations has even stronger property:
the action of the Galois group is realized as a third order linear fractional
transformation.
More precisely, we have the following.
\begin{lem}
\label{InvLem}
(i) 
The  equation~\eqref{Cubic1} is invariant under the linear-fractional transformation
\begin{equation}
\label{LF1}
x\mapsto\frac{x-1}{x}.
\end{equation}

(ii)
The  equation~\eqref{Cubic2} is invariant under the linear-fractional transformation
\begin{equation}
\label{LF2}
x\mapsto-\frac{x+1}{x}.
\end{equation}

\end{lem}

\begin{proof}
This is a straightforward computation.
\end{proof}

Note that the linear-fractional transformations~\eqref{LF1} and~\eqref{LF2} 
correspond to the action of the following third order elements of~$\PSL(2,\Z)$
$$
TS=
\begin{pmatrix}
1&-1\\[4pt]
1&0
\end{pmatrix}
\qquad\hbox{and}\qquad
T^{-1}S=
\begin{pmatrix}
1&1\\[4pt]
-1&0
\end{pmatrix},
$$
respectively.

This invariance implies that the transformations~\eqref{LF1} and~\eqref{LF2}
act on the roots of the equations.
Since they are third order maps, they have to cyclically permute the roots.
We obtain the following.

\begin{cor}
\label{CorLem}
(i) 
Up to renumeration, the roots of
the  equation~\eqref{Cubic1} are related to each other by the following relations
\begin{equation}
\label{QRel1}
x_1=\frac{x_2-1}{x_2},
\qquad
x_2=\frac{x_3-1}{x_3},
\qquad
x_3=\frac{x_1-1}{x_1}.
\end{equation}

(ii) 
Up to renumeration, the roots of
the  equation~\eqref{Cubic2} are related to each other by the following relations
\begin{equation}
\label{QRel2}
x_1=-\frac{x_2+1}{x_2},
\qquad
x_2=-\frac{x_3+1}{x_3},
\qquad
x_3=-\frac{x_1+1}{x_1}.
\end{equation}
\end{cor}

The roots of the polynomials~\eqref{Cubic1} and~\eqref{Cubic2} are therefore
$\PSL(2,\Z)$-equivalent numbers.

\begin{rem}
(1)
It is well-known that the elements $TS$ and $T^{-1}S$ are the only,
up to conjugation, third order elements of~$\PSL(2,\Z)$.
These are representatives of two conjugacy classes of~$\PSL(2,\Z)$ with trace~$1$
(see, e.g~\cite{Cho}).
Therefore, \eqref{Cubic1} and~\eqref{Cubic2} are the only, 
up to $\PSL(2,\Z)$-equivalence, cubic equations
whose roots are $\PSL(2,\Z)$-equivalent to each other.

(2)
Note also that the property that the discriminant of a cubic equation is a rational square
does not imply that the roots are $\PSL(2,\Z)$-equivalent.
For instance, the discriminant of the polynomial $x^3-12x+8$ is equal to~$72^2$,
and the polynomial $ x^3-13x+13$ has discriminant $65^2$,
but the roots are not equivalent for different reasons. 
The roots of the first polynomial, $x^3-12x+8$, are $x_k=4\cos \left(\frac{2\pi k}{9 }\right)$ 
(where~$k=1,2,3$) such that 
$$
x_{1}=\frac{-x_2+2}{2x_2 },
\qquad x_{2}=\frac{-x_3+2}{2x_3 },
\qquad x_{3}=\frac{-x_1+2}{2x_1 }.
$$
They do not belong to the same $\PSL(2,\mathbb{Z})$-orbit because the matrix
$\bigl(\begin{smallmatrix} -1&2\\
\;2&0 \end{smallmatrix}\bigr)$
does not belong to $\PSL(2,\mathbb{Z})$. 
The second polynomial, $x^3-13x+13$, has the roots 
$$
x_k=2\sqrt{\frac{13}{3 }}
\cos \left(\frac{1}{ 3}\arccos\left(-\frac{3}{2}\sqrt{\frac{3}{13 }} \right)
+\frac{2\pi k}{3}\right)\qquad (k=1,2,3)
$$ 
which are not $\PSL(2,\Z)$-equivalent because they are not connected by any linear-fractional transformation:
$$
x_{1,2}=-\frac{x_3}{2 }\pm \sqrt{13-\frac{3x_3^2}{ 4}}.
$$
\end{rem}

\subsection{The $q$-deformed Galois group action}
The $q$-deformations of the elements $TS$ and $T^{-1}S$
are given by
$$
T_qS_q=
\begin{pmatrix}
1&-1\\[4pt]
1&0
\end{pmatrix}
\qquad\hbox{and}\qquad
T_q^{-1}S_q=
\begin{pmatrix}
q&1\\[4pt]
-q^2&0
\end{pmatrix}.
$$
Therefore, the property of $\PSL(2,\Z)$-invariance of $q$-deformed irrationals implies 
the next statement.

\begin{prop}
\label{TransProp}
The $q$-deformed roots 
$X_1(q)=\left[x_1\right]_q,X_2(q)=\left[x_2\right]_q,X_3(q)=\left[x_3\right]_q$
has to be related by the following $q$-analogues of~\eqref{QRel1} and~\eqref{QRel2}:
\begin{equation}
\label{QRel1Bis}
X_1(q)=\frac{X_2(q)-1}{X_2(q)},
\qquad
X_2(q)=\frac{X_3(q)-1}{X_3(q)},
\qquad
X_3(q)=\frac{X_1(q)-1}{X_1(q)},
\end{equation}
for~\eqref{Cubic1}, and
\begin{equation}
\label{QRel2Bis}
X_1(q)=-\frac{qX_2(q)+1}{q^2X_2(q)},
\qquad
X_2(q)=-\frac{qX_3(q)+1}{q^2X_3(q)},
\qquad
X_3(q)=-\frac{qX_1(q)+1}{q^2X_1(q)},
\end{equation}
for~\eqref{Cubic2}.
\end{prop}

\subsection{Vieta's formulas}

Let us now prove a statement 
which is an equivalent reformulation of Theorem~\ref{FirstThm}.

\begin{thm}
\label{VietaThm}
(i)
The $q$-deformations of the roots of~\eqref{Cubic1} 
satisfy the Vieta type equations
\begin{equation}
\label{Vieta1}
\begin{array}{rcl}
X_1(q)X_2(q)X_3(q)&=&-1,\\[4pt]
X_1(q)X_2(q)+X_2(q)X_3(q)+X_3(q)X_1(q)&=&X_1(q)+X_2(q)+X_3(q)-3.
\end{array}
\end{equation}

(ii)
The $q$-deformations of the roots of~\eqref{Cubic2} 
satisfy the Vieta type equations
\begin{equation}
\label{Vieta2}
\begin{array}{rcl}
X_1(q)X_2(q)X_3(q)&=&q^{-3},\\[4pt]
X_1(q)X_2(q)+X_2(q)X_3(q)+X_3(q)X_1(q)&=&
q^{-1}\left(X_1(q)+X_2(q)+X_3(q)\right)-3q^{-2}.
\end{array}
\end{equation}
\end{thm}

\begin{proof}
Part (i). 
The second and the third formulas~\eqref{QRel1Bis} imply
$$
X_2(q)=-\frac{1}{X_1(q)-1}.
$$
Therefore,
$$
X_1(q)X_2(q)X_3(q)=-X_1(q)\frac{1}{X_1(q)-1}\frac{X_1(q)-1}{X_1(q)}=-1,
$$
hence the first formula~\eqref{Vieta1}.
One again applying~\eqref{QRel1Bis}, we have
$$
\left(X_1(q)-1\right)\left(X_2(q)-1\right)\left(X_3(q)-1\right)=
\left(X_1(q)X_2(q)X_3(q)\right)^2=1,
$$
which readily gives the second formula~\eqref{Vieta1}.

Part (ii) is a similar consequence of~\eqref{QRel2Bis}
and we omit the details.
\end{proof}

Theorem~\ref{FirstThm} is obviously equivalent to Theorem~\ref{VietaThm}.
Indeed, rewriting the $q$-deformed equation in terms of its roots
$$
\left(X-X_1(q)\right)\left(X-X_2(q)\right)\left(X-X_3(q)\right)=0,
$$
and applying~\eqref{Vieta1} and~\eqref{Vieta2}, 
we immediately have~\eqref{Cubic1Quant}
and~\eqref{Cubic2Quant}, respectively.
Theorem~\ref{FirstThm} is proved.

\begin{rem}
To conclude this discussion, one can say that two of three 
Vieta's formulas of~\eqref{Cubic1} and~\eqref{Cubic2} are preserved by the $q$-deformation.
Computer experiments convinced us that
this phenomenon 
is true uniquely for the equations that belong to the families~\eqref{Cubic1} and~\eqref{Cubic2}.
It is a challenging problem to figure out if something similar holds for any other cubic equation.
\end{rem}

\subsection{More about the coefficients $B_\mp(q)$}

Let us discuss the nature of the coefficients $B_-(q)$ and~$B_-(q)$
of the equations~\eqref{Cubic1Quant} and~\eqref{Cubic2Quant}.
Besides the obvious expression $B_\mp(q) = X_1(q)+X_2(q)+X_3(q)$,
there is another way to calculate it.

For convenience, let us use the following notation.
Looking at the polynomial~\eqref{Cubic1} as an operator, we write 
\begin{equation}
\label{E1}
\E_-\left(X\right):=
\frac{X^3-bX+\left(b-3\right)X+1}{X\left(X-1\right)}.
\end{equation}
Clearly, $\E_-\left(x_1\right)=\E_-\left(x_2\right)=\E_-\left(x_3\right)=0$,
but after the $q$-deformation this is no longer the case, 
$\E_-\left(X_i(q)\right)\not=0$, for~$i=1,2,3$.
However, it turns out surprisingly that these three expressions coincide.

In the case~\eqref{Cubic2}, we write
\begin{equation}
\label{E2}
\E_+\left(X\right):=
\frac{X^3 - bX^2 - \left(q^{-1}b+3q^{-2}\right)X - q^{-3}}{X\left(X+q^{-1}\right)}.
\end{equation}
One again, $\E_+\left(X_i(q)\right)\not=0$, but these expressions
turns out to coincide, for~$i=1,2,3$.

More precisely, we have the following.

\begin{prop}
\label{BProp}
The solutions of the equation~\eqref{Cubic1Quant} and~\eqref{Cubic2Quant} satisfy
\begin{equation}
\label{Defect1}
\E_\mp\left(X_1(q)\right)=
\E_\mp\left(X_2(q)\right)=
\E_\mp\left(X_3(q)\right)\;=\;
B_\mp(q)-b,
\end{equation}
where $\E_\mp$ is as in~\eqref{E1} and~\eqref{E2}, respectively.
\end{prop}

\begin{proof}
Starting from $B_-(q) = X_1(q)+X_2(q)+X_3(q)$ and
applying~\eqref{QRel1Bis}, we have
$$
B_-(q)=
X_1(q)-\frac{1}{X_1(q)-1}+\frac{X_1(q)-1}{X_1(q)}=
\frac{X_1(q)^3-3X_1(q)+1}{X_1(q)\left(X_1(q)-1\right)},
$$
and similarly for $X_2(q)$ and $X_3(q)$.
Subtracting $b$ gives the first case of~\eqref{Defect1}.

Similarly, applying~\eqref{QRel2Bis} we have
$$
B_+(q)=
X_1(q)-\frac{1}{q^2X_1(q)+q}-\frac{qX_1(q)+1}{q^2X_1(q)}=
\frac{X_1(q)^3-3q^{-2}X_1(q)-q^{-3}}{X_1(q)\left(X_1(q)+q^{-1}\right)},
$$
as well as for $X_2(q)$ and $X_3(q)$.
One again, subtracting~$b$ we obtain the second case of~\eqref{Defect1}.
\end{proof}

\section{Main examples}\label{HeptaSec}

In this final section, we give more details for the $q$-deformed equations~\eqref{HeptaEq}
and~\eqref{NonaEq}.
Our goal is to analyze the coefficient~$B_\mp(q)$ in this case.

Let us also mention that another surprising appearance of
the regular heptagon in the context of $q$-deformations
can be found in~\cite{Hig} (see Example~5.7).

\subsection{The ``quantum heptagon''}

Let us give more details in the case of the 
$q$-deformed equation~\eqref{HeptaEq} with the roots~\eqref{HeptaRoots}.
Vieta's formulas~\eqref{Vieta2} can now be checked with the computer.
After a computation, the $q$-deformed roots are the following Laurent series in~$q$
$$
\begin{array}{rcl}
X_1(q) &=&
1+q^5 -q^7 -q^8 +2q^{10} +2q^{11} -q^{12} -4q^{13} 
-3q^{14} +3q^{15} +8q^{16}+4q^{17} -8q^{18} -15q^{19}\\[4pt]
&&  -4q^{20} +19q^{21} +27q^{22} -42q^{24} -47q^{25} 
+16q^{26} +91q^{27} +71q^{28} -62q^{29} -178q^{30}\cdots\\[8pt]
X_2(q) &=&
-q^{-1} +1-q+q^2 -q^3 +q^4 -q^6 +q^7 -2q^8 +3q^9 -2q^{10} +2q^{11} 
-2q^{12} -q^{13} +2q^{14}\\[4pt]
&& -2q^{15} +6q^{16} - 
6q^{17} +3q^{18} -6q^{19}+3q^{20} +5q^{21} -2q^{22} +6q^{23}  \\[4pt]
&&
-17q^{24} +7q^{25} -4q^{26} +16q^{27} +2q^{28} -
13q^{29} -3q^{30}\cdots\\[8pt]
X_3(q)&=&
-q^{-2}-q^{-1}+q^3 -q^5 -q^6 +q^8 +2q^9 +q^{10} -2q^{11} 
-4q^{12} -2q^{13} +3q^{14}  \\[4pt]
&&+7q^{15} +5q^{16} -4q^{17} -13q^{18} 
-11q^{19} +5q^{20} +23q^{21} +23q^{22} -5q^{23} \\[4pt]
&& -39q^{24} -44q^{25}
 -4q^{26} +63q^{27} +93q^{28} +29q^{29} -114q^{30} \cdots
\end{array}
$$
The equation $X_1(q)X_2(q)X_3(q)=q^{-3}$ and the second equation~\eqref{Vieta2}
are then obtained as a spectacular cancellation of terms in infinite series.

The coefficient $B_+(q)=X_1(q)+X_2(q)+X_3(q)$ 
of the $q$-deformed equation~\eqref{Cubic2Quant} is in this case
$$
\begin{array}{rcl}
B_+(q)&=&-q^{-2}-2q^{-1}+2-q+q^2 +q^4 -2q^6 -2q^8 +5q^9 +q^{10} +2q^{11} 
-7q^{12} -7q^{13} +2q^{14} \\[4pt]
&&+8q^{15} +19q^{16} -
6q^{17} -18q^{18} -32q^{19} +4q^{20} +47q^{21} +48q^{22} 
+q^{23} -98q^{24} \\[4pt]
&&-84q^{25} +8q^{26} +170q^{27} +
166q^{28} -46q^{29} -295q^{30}\cdots
\end{array}
$$

Our next goal is to analyze the function~$B_+(q)$.
The coefficients of this series 
look like values of an oscillating function with exponentially growing amplitude.
We are convinced that this series does not end, see Fig.~1,
even though we cannot prove it.

\begin{center}
\begin{tikzpicture}[scale=.13]

\draw [->] (-3,0)--(95,0);
\draw [->] (0,-16)--(0,16);

\draw [thick] (1, 0.)--(2, -0.693147)--(3, 0.693147)--(4, 0.)--(5, 0.)--(6, 
  0)--(7, 0.)--(8, 0)--(9, -0.693147)--(10, 
  0)--(11, -0.693147)--(12, 1.60944)--(13, 0.)--(14, 
  0.693147)--(15, -1.94591)--(16, -1.94591)--(17, 0.693147)--(18, 
  2.07944)--(19, 
  2.94444)--(20, -1.79176)--(21, -2.89037)--(22, -3.46574)--(23, 
  1.38629)--(24, 3.85015)--(25, 3.8712)--(26, 
  0.)--(27, -4.58497)--(28, -4.43082)--(29, 2.07944)--(30, 
  5.1358)--(31, 
  5.11199)--(32, -3.82864)--(33, -5.68698)--(34, -5.68358)--(35, 
  4.80402)--(36, 6.27288)--(37, 
  6.07535)--(38, -5.50126)--(39, -6.79571)--(40, -6.44413)--(41, 
  6.20254)--(42, 7.25771)--(43, 
  6.82437)--(44, -7.03439)--(45, -7.67694)--(46, -6.71296)--(47, 
  7.73456)--(48, 
  7.91059)--(49, -6.1334)--(50, -8.11193)--(51, -7.87891)--(52, 
  8.02519)--(53, 8.55101)--(54, 
  7.16472)--(55, -9.22779)--(56, -9.32679)--(57, 9.17906)--(58, 
  10.3327)--(59, 9.53734)--(60, -10.7392)--(61, -11.0588)--(62, 
  7.20266)--(63, 11.5611)--(64, 
  11.7148)--(65, -9.32723)--(66, -12.3435)--(67, -12.6274)--(68, 
  11.0814)--(69, 13.3518)--(70, 
  13.3018)--(71, -12.9909)--(72, -14.1099)--(73, -13.4883)--(74, 
  13.7957)--(75, 14.579)--(76, 
  14.2385)--(77, -14.2023)--(78, -15.4566)--(79, -15.0207)--(80, 
  15.5044)--(81, 16.1697)--(82, 
  14.7705)--(83, -16.3637)--(84, -16.6241)--(85, 14.9505)--(86, 
  16.8314)--(87, 16.3562)--(88, 
  16.4214)--(89, -16.5705)--(90, -18.4451)--(91, 15.5667);

\draw [below] (45,-18.2) node {\textsc{Figure 1.} 
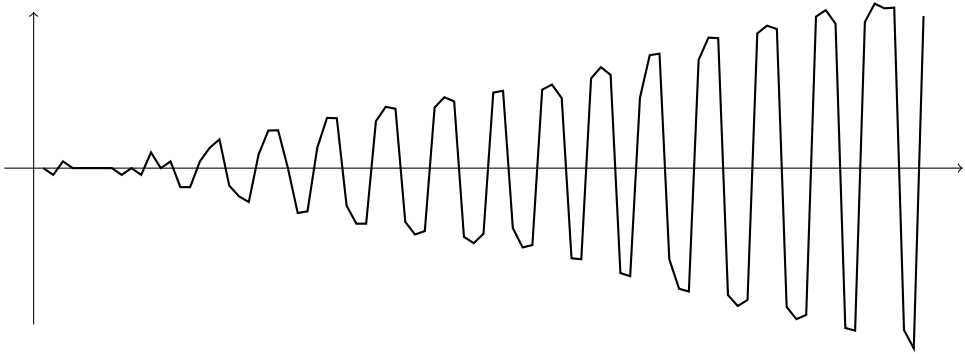
The first $91$ coefficients of the function $B_+(q)$ on a logarithmic scale. };

\end{tikzpicture}
\end{center}

A natural tool for analyzing infinite series $f(z)=c_0+c_1q+c_2q^2+\ldots$ is the so-called \textit{$C$-table}
$$
\begin{array}{|cccc|}
\hline C(0 / 0) & C(1 / 0) & C(2 / 0) & \ldots \\
C(0 / 1) & C(1 / 1) & C(2 / 1) & \ldots \\
C(0 / 2) & C(1 / 2) & C(2 / 2) & \ldots \\
\vdots & \vdots & \vdots & \ddots \\
\hline
\end{array}\,,
$$
which go back to the theory of Pad\'e approximants, see.~\cite{BGM}. 
The $C$-table consists of Hankel determinants
$$C(L/M)=\left|\begin{array}{cccc}
c_{L-M+1} & c_{L-M+2} & \cdots & c_L \\
c_{L-M+2} & c_{L-M+3} & \cdots & c_{L+1} \\
\vdots & \vdots & & \vdots \\
c_L & c_{L+1} & \cdots & c_{L+M-1}
\end{array}\right|\qquad(L,M\ge 0).$$
By definition $C(L/0)=1$, 
the next values are $C(L/1)=c_L$, $C(0/M)=(-1)^{M(M-1)/2}c_0^M$, 
and $c_k:=0$ for $k<0$. 

Rational functions are characterized by the property that the $C$-table contains an infinite block of zeros: 
$C(\lambda+i/\mu+j)=0$ for some $\lambda,\mu\ge 0$ and $i,j=0,1,2, \ldots$

\begin{ex}
(i)
The function 
$$
\left[\frac{5}{3}\right]_q=\frac{1+q+2q^2+q^3}{ 1+q+q^2}=1+q^2-q^4+q^5-q^7+q^8-q^{10}+q^{11}\ldots
$$
(see~\cite{MGOfmsigma}) has the following $C$-table:

\begin{center}
\begin{tabular}{|c|ccccccccccc|}
\hline 
$M\backslash L$
 & $0$ & $1$ & $2$ & $3$ & $4$ & $5$ & $6$ & $7$ & $8$ & $9$ & \\
\hline $0$ & $1$ & $1$ & $1$ & $1$ & $1$ & $1$ & $1$ & $1$ & $1$ & $1$ & $\cdots$ \\
$1$ & $1$ & $0$ & $1$ & $0$ & $-1$ & $1$ & $0$ & $-1$ & $1$ & $0$ & $\cdots$ \\
$2$ & $-1$ & $1$ & $-1$ & $-1$ & $-1$ & $-1$ & $-1$ & $-1$ & $-1$ & $-1$ & $\cdots$ \\
$3$ & $-1$ & $0$ & $2$ & $-1$ & $0$ & $0$ & $0$ & $0$ & $0$ & $0$& $\cdots$ \\
$4$ & $1$ & $2$ & $4$ & $1$ & $0$ & $0$ & $0$ & $0$ & $0$ & $0$ & $\cdots$ \\
$5$ & $1$ & $1$ & $7$ & $1$ & $0$ & $0$ & $0$ & $0$ & $0$ & $0$ & $\cdots$ \\
 $\vdots$ & $\vdots$ & $\vdots$ & $\vdots$ & $\vdots$&$\vdots$ &$\vdots$ & $\vdots$& $\vdots$&$\vdots$ & $\vdots$ &$\ddots$ \\
\hline
\end{tabular}
\end{center}

(ii)
On the other hand,
the $C$-tables of some $q$-deformed quadratic irrationals have some sort of regular behavior. 
For example, the $C$-table of the $q$-deformed golden ratio~\cite{MGOexp,LMGadv}
$$[\varphi]_q=\frac{q^2+q-1+\sqrt{q^4+2 q^3-q^2+2 q+1}}{2 q}$$
has periodic structure in the neighborhood of the main diagonal, see~\cite{OP}:
\begin{center}
\begin{tabular}{|c|ccccccccccccccc|}
\hline 
$M\backslash L$
 & $0$ & $1$ & $2$ & $3$ & $4$ & $5$ & $6$ & $7$ & $8$ & $9$ & $10$ & $11$ & $12$ &$13$ &$14$  \\
\hline $0$ & $1$ & $1$ & $1$ & & & & & & & & & & & & \\
$1$ & $1$ & $0$ & $1$ & $-1$ &  &  & & & & & &&&& \\
$2$ & $-1$ & $1$ & $-1$ & $1$ & $0$ & $0$ & & & & & &&&& \\
$3$ & $-1$ & $1$ & $0$ & $1$ & $0$ & $0$ & & & & & &&&& \\
$4$ & & $-1$ & $-1$ & $-1$ & $-1$ & $-1$ & $-1$ & & & & &&&& \\
$5$ & & & $0$ & $0$ & $-1$ & $0$ & $-1$ & $1$ & &  & &&&& \\
$6$ & & & $0$ & $0$ & $1$ & $-1$ & $1$ & $-1$ & $0$ & $0$ &&&&&  \\
$7$ & & &  & & $1$ & $-1$ & $0$ & $-1$ & $0$ & $0$ &  &&&&\\
$8$ & & &  & & & $1$ & $1$ & $1$ & $1$ & $1$ & $1$&&&& \\
$9$ & & & &  & & & $0$ & $0$ & $1$ & $0$ & $1$ & $-1$&&& \\
$10$ & & & &  & & & $0$ & $0$ & $-1$ & $1$ & $-1$ & $1$ & $0$ & $0$&\\
$11$ & & & & & &  & & & $-1$ & $1$ & $0$ & $1$ & $0$ & $0$ &  \\
$12$ & & & & & & &  & & & $-1$ & $-1$ & $-1$ & $-1$ & $-1$ & $-1$  \\

\hline
\end{tabular}
\end{center}

\end{ex}
\medskip

Computer experiments show that for the series $B_+(q)$
Hankel determinants $C(L/M)$ grow fast. 
For example, the values of $C(L/M)$ for $1\le L,M\le 12$ are  shown in the following table. 

\medskip
\begin{tiny}
\begin{center}
\begin{tabular}{|c|cccccccccccc|}
\hline 
$M\backslash L$
 & $1$ & $2$ & $3$ & $4$ & $5$ & $6$ & $7$ & $8$ & $9$ & $10$ & $11$ & $12$  \\
\hline $1$ & $-2$ & $ 2$ & $ -1$ & $ 1$ & $ 0$ & $ 1$ & $ 0$ & $ -2$ & $ 0$ & $ -2$ & $ 5$ & $ 1$  \\
$2$ & $-6$ & $ -2$ & $ 1$ & $ -1$ & $ 1$ & $ -1$ & $ -2$ & $ -4$ & $ 4$ & $ -4$ & $ -27$ & $ 9$\\
$3$ & $17$ & $ -5$ & $ -1$ & $ 0$ & $ 1$ & $ -3$ & $ 0$ & $ 12$ & $ -20$ & $ 62$ & $ -153$ & $ 216$ \\
$4$ & $49$ & $ 21$ & $ -1$ & $ 1$ & $ -1$ & $ 9$ & $ 18$ & $ 36$ & $ 86$ & $ 196$ & $ 371$ & $ 664$ \\
$5$ & $-140$ & $ 98$ & $ -20$ & $ 7$ & $ 8$ & $ 33$ & $ 9$ & $ 21$ & $ 17$ & $ -105$ & $ 49$ & $ 652$ \\
$6$ & $-402$ & $ -324$ & $ -286$ & $ -209$ & $ -167$ & $ -113$ & $ 34$ & $ -8$ & $ -29$ & $ -52$ & $ -191$ & $ -629 $\\
$7$ &$ 1152$ & $ 102$ & $ 704$ & $ 583$ & $ -534$ & $ -559$ & $ -28$ & $ -50$ & $ -25$ & $ -27$ & $ -77$ & $ -163$ \\
$8$ & $3301$ & $ -2471$ & $ 1525$ & $ 3425$ & $ 3659$ & $ 2633$ & $ 799$ & $ 225$ & $ -25$ & $ -23$ & $ 8$ & $ 24$\\
$9$ & $-9461$ & $ -10508$ & $ -15325$ & $ -10550$ & $ 8184$ & $ 7172$ & $ 1642$ & $ 1412$ & $ 232$ & $ 27$ & $ 8$ & $ -16$ \\
$10$ & $-27113$ & $ -13991$ & $ -81309$ & $ -69116$ & $ -38984$ & $ -14432$ & $ 9300$ & $ -7168$ & $ 628$ & $ -49$ & $ -62$ & $ 156$\\
$11$ & $77698$ & $ -191167$ & $ 368297$ & $ 152348$ & $ -63816$ & $ -79592$ & $ 10328$ & $ -32252$ & $ -186$ & $ 
-1531$ & $ -1436$ & $ -556$ \\
$12$ & $222667$ & $ 566713$ & $ 2026425$ & $ 675866$ & $ 415508$ & $ 484616$ & $ 264552$ & $ 145384$ & $ 78572$ & $ 42385$ & $ 19530$ & $ 5272$ \\

\hline
\end{tabular}
\end{center}
\end{tiny}

\medskip

\noindent
This behavior clearly indicates that the function~$B_+(q)$ is not a rational function.

\subsection{The ``quantum nonagon''}

The quantum version of the equation~\eqref{NonaEq} with the roots~\eqref{NonaRoots} is similar. 
Computer computation of the $q$-deformed roots give the following Laurent series in~$q$
$$
\begin{array}{rcl}
X_1(q) &=&
1+q^2-q^3+q^4-q^5+q^6-q^7+q^8-q^9+2 q^{10}-3 q^{11}+3 q^{12}-5 q^{13}+8 q^{14}\\[4pt]
&&-7 q^{15}+5 q^{16}-9 q^{17}+13 q^{18}-3 q^{19}-9 q^{20}-2 q^{21}+
5q^{22}+42 q^{23}-71 q^{24}\\[4pt]
&&+22 q^{25}-39 q^{26}+202 q^{27}-217 q^{28}-136 q^{30}\cdots\\[8pt]
X_2(q) &=&
q^2-q^3+q^5-q^6+q^8-q^9+q^{10}-q^{11}-q^{12}+2 q^{13}+q^{14}-q^{15}-5 q^{16}\\[4pt]
&&+3 q^{17}+10 q^{18}-4 q^{19}-26 q^{20}+20 q^{21}+28 q^{22}-9q^{23}-69 q^{24}\\[4pt]
&&+11 q^{25}+122 q^{26}+21 q^{27}-236 q^{28}-107 q^{29}+471 q^{30}\cdots\\[8pt]
X_3(q)&=&
-\frac{1}{q^2}-\frac{1}{q}+q^6-q^8-2 q^9+q^{10}+4 q^{11}-q^{12}-6 q^{13}-q^{14}+15 q^{15}-q^{16}\\[4pt]
&&-20 q^{17}-14 q^{18}+44 q^{19}+24 q^{20}-60
   q^{21}-70 q^{22}+97 q^{23}+161 q^{24}\\[4pt]
&&-155 q^{25}-305 q^{26}+139 q^{27}+697 q^{28}-185 q^{29}-1192 q^{30} \cdots
\end{array}
$$

Vieta's formulas~\eqref{Vieta1} again express non-trivial arithmetic properties of series $X_k(q)$.
The coefficient $B_-(q)=X_1(q)+X_2(q)+X_3(q)$ 
of the $q$-deformed equation~\eqref{Cubic1Quant} is in this case
$$
\begin{array}{rcl}
B_-(q)&=&-\frac{1}{q^2}-\frac{1}{q}+1+2 q^2-2 q^3+q^4+q^6-q^7+q^8-4 q^9+4 q^{10}+q^{12}-9 q^{13}+8 q^{14}\\[4pt]
&&+7 q^{15}-q^{16}-26 q^{17}+9 q^{18}+37 q^{19}-11
   q^{20}-42 q^{21}-37 q^{22}+130 q^{23}+21 q^{24}\\[4pt]
   &&-122 q^{25}-222 q^{26}+362 q^{27}+244 q^{28}-292 q^{29}-857 q^{30}\cdots
\end{array}
$$

The coefficients of the series $B_-(q)$
show a similar behavior as the coefficients of the series $B_+(q)$; see Fig.~2.

\begin{center}
\begin{tikzpicture}[scale=.05]

\draw [->] (-3,0)--(255,0);
\draw [->] (0,-75)--(0,75);

\draw [thick] (1, 0.)--(2, 0.)--(3, 0.)--(4, 0)--(5, 
  0.693147)--(6, -0.693147)--(7, 0.)--(8, 0)--(9, 
  0.)--(10, 0.)--(11, 0.)--(12, -1.38629)--(13, 1.38629)--(14, 
  0)--(15, 0.)--(16, -2.19722)--(17, 2.07944)--(18, 
  1.94591)--(19, 0.)--(20, -3.2581)--(21, 2.19722)--(22, 
  3.61092)--(23, -2.3979)--(24, -3.73767)--(25, -3.61092)--(26, 
  4.86753)--(27, 3.04452)--(28, -4.80402)--(29, -5.40268)--(30, 
  5.89164)--(31, 5.49717)--(32, -5.67675)--(33, -6.75344)--(34, 
  6.56667)--(35, 7.12367)--(36, -6.02828)--(37, -7.97212)--(38, 
  6.43294)--(39, 8.42595)--(40, 
  6.82871)--(41, -9.06056)--(42, -8.12059)--(43, 9.55733)--(44, 
  9.13594)--(45, -9.95916)--(46, -10.1151)--(47, 10.4837)--(48, 
  10.6894)--(49, -10.5327)--(50, -11.5528)--(51, 11.0298)--(52, 
  11.9916)--(53, -9.76308)--(54, -12.7136)--(55, -7.67276)--(56, 
  13.0958)--(57, 12.3259)--(58, -13.6228)--(59, -13.2199)--(60, 
  13.9366)--(61, 14.0651)--(62, -14.1239)--(63, -14.7935)--(64, 
  14.1788)--(65, 15.2817)--(66, 
  12.3905)--(67, -15.85)--(68, -14.879)--(69, 15.9823)--(70, 
  16.3584)--(71, -16.0794)--(72, -17.0859)--(73, -15.5554)--(74, 
  17.8402)--(75, 17.6263)--(76, -18.3034)--(77, -18.8837)--(78, 
  18.5875)--(79, 
  19.8407)--(80, -18.1722)--(81, -20.6595)--(82, -19.2294)--(83, 
  21.4097)--(84, 20.8739)--(85, -22.0543)--(86, -22.0695)--(87, 
  22.6479)--(88, 23.0383)--(89, -23.1182)--(90, -23.9275)--(91, 
  23.4646)--(92, 
  24.7322)--(93, -23.4559)--(94, -25.4885)--(95, -21.4257)--(96, 
  26.1936)--(97, 24.9841)--(98, -26.8536)--(99, -26.3225)--(100, 
  27.4682)--(101, 27.3717)--(102, -28.0279)--(103, -28.2956)--(104, 
  28.5248)--(105, 29.137)--(106, -28.9235)--(107, -29.9239)--(108, 
  29.1576)--(109, 
  30.6633)--(110, -28.9187)--(111, -31.3638)--(112, -28.9472)--(113, 
  32.0264)--(114, 30.9736)--(115, -32.6515)--(116, -32.1779)--(117, 
  33.2363)--(118, 33.1672)--(119, -33.7735)--(120, -34.0494)--(121, 
  34.2507)--(122, 34.8623)--(123, -34.6395)--(124, -35.625)--(125, 
  34.8736)--(126, 
  36.346)--(127, -34.6978)--(128, -37.0308)--(129, -34.268)--(130, 
  37.6811)--(131, 36.5215)--(132, -38.2969)--(133, -37.745)--(134, 
  38.8755)--(135, 38.7365)--(136, -39.411)--(137, -39.616)--(138, 
  39.8923)--(139, 40.4254)--(140, -40.2971)--(141, -41.1841)--(142, 
  40.5747)--(143, 41.902)--(144, -40.5639)--(145, -42.5844)--(146, 
  37.8425)--(147, 43.2339)--(148, 
  41.7886)--(149, -43.8506)--(150, -43.1305)--(151, 44.4329)--(152, 
  44.1609)--(153, -44.9761)--(154, -45.0587)--(155, 45.4718)--(156, 
  45.8784)--(157, -45.9036)--(158, -46.6438)--(159, 46.237)--(160, 
  47.367)--(161, -46.3807)--(162, -48.0543)--(163, 45.8993)--(164, 
  48.7092)--(165, 46.6327)--(166, -49.3327)--(167, -48.3283)--(168, 
  49.9239)--(169, 49.4495)--(170, -50.4801)--(171, -50.3886)--(172, 
  50.9951)--(173, 51.2322)--(174, -51.4579)--(175, -52.0135)--(176, 
  51.8467)--(177, 52.7487)--(178, -52.1126)--(179, -53.4464)--(180, 
  52.1044)--(181, 
  54.1113)--(182, -50.1276)--(183, -54.7455)--(184, -53.2088)--(185, 
  55.3494)--(186, 54.5594)--(187, -55.9216)--(188, -55.5831)--(189, 
  56.4585)--(190, 56.4712)--(191, -56.9532)--(192, -57.2806)--(193, 
  57.3933)--(194, 58.036)--(195, -57.7539)--(196, -58.75)--(197, 
  57.9769)--(198, 
  59.4293)--(199, -57.865)--(200, -60.0779)--(201, -56.2862)--(202, 
  60.6975)--(203, 59.3032)--(204, -61.2882)--(205, -60.5538)--(206, 
  61.8487)--(207, 61.5366)--(208, -62.3756)--(209, -62.3988)--(210, 
  62.8629)--(211, 63.1888)--(212, -63.2996)--(213, -63.9282)--(214, 
  63.6645)--(215, 64.6285)--(216, -63.911)--(217, -65.2963)--(218, 
  63.8984)--(219, 
  65.9353)--(220, -62.3576)--(221, -66.5478)--(222, -64.8437)--(223, 
  67.1344)--(224, 66.1999)--(225, -67.6951)--(226, -67.2086)--(227, 
  68.2284)--(228, 68.0778)--(229, -68.7319)--(230, -68.868)--(231, 
  69.2009)--(232, 69.6052)--(233, -69.6277)--(234, -70.3033)--(235, 
  69.9989)--(236, 70.9704)--(237, -70.2879)--(238, -71.6121)--(239, 
  70.4313)--(240, 
  72.2324)--(241, -70.2033)--(242, -72.8346)--(243, -69.2794)--(244, 
  73.4217)--(245, 71.6939)--(246, -73.9969)--(247, -72.8538)--(248, 
  74.5635)--(249, 73.7736)--(250, -75.1254);

\draw [below] (125,-75.2) node {\textsc{Figure 2.} 
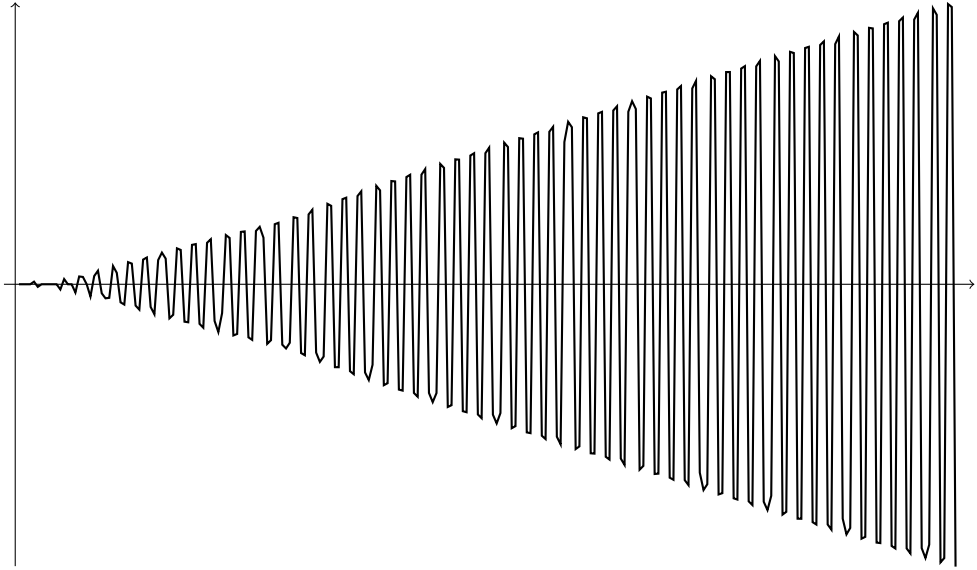
The first $250$ coefficients of the function $B_-(q)$ on a logarithmic scale. };

\end{tikzpicture}
\end{center}

The Hankel determinants $C(L/M)$ for the series $B_-(q)$ grow fast as in the case of  $B_+(q)$. This observation allows us to state a conjecture that $B_-(q)$ is not rational function in this case.

\bigbreak \noindent
{\bf Acknowledgements}.
We are grateful to Sophie Morier-Genoud for enlightening discussions.


\begin{thebibliography}{99}

\bibitem{BGM}
{\sc Baker Jr., G. A. and Graves-Morris, P.}
\newblock Pad\'e approximants. {P}art {I}
\newblock {\em  
Encyclopedia of Mathematics and its Applications, Vol. 13.} Addison-Wesley Publishing Co., Reading, Mass. (1981).

\bibitem{BBL}
{\sc Bapat, A., Becker, L., and Licata, A.}
\newblock $q$-deformed rational numbers and the 2-{C}alabi--{Y}au category of
  type $A_2$.
\newblock {\em Forum Math. Sigma 11\/} (2023), Paper No. e47.

\bibitem{Bir}
{\sc Birman, J.} 
\newblock Braids, links, and mapping class groups. 
\newblock Ann. Math. Stud. 82. 
Princeton Univ. Press, 1974.

\bibitem{Hig}
{\sc Burcroff, A., Ovenhouse, N., Schiffler, R., and W. Zhang, S.} 
\newblock Higher q-Continued Fractions.
\newblock ArXiv:2408.06902.


\bibitem{Cho}
{\sc 
Chowla, S, Cowles J., and Cowles M.}
\newblock On the number of conjugacy classes in $\SL(2,\Z)$.
\newblock {\em Journal of Number Theory 12\/}  (1980), 372--377.

\bibitem{Perrine}
{\sc Jouteur, P.}
\newblock Burau representation of $\B_4$ and quantization of the rational projective plane.
\newblock ArXiv:2407.20645.

\bibitem{KW}
{\sc Kogiso, T., and Wakui, M.}
\newblock A bridge between Conway-Coxeter friezes and rational tangles through the Kauffman bracket polynomials. 
\newblock {\em J. Knot Theory Ramifications 28\/} (2019), no. 14, 40 pp.

\bibitem{Las}
{\sc Lasker, J.}
\newblock First and Second Derivatives of the $q$-Rationals.
\newblock ArXiv:2309.10819.

\bibitem{LMGadv}
{\sc Leclere, L., and Morier-Genoud, S.}
\newblock {$q$}-deformations in the modular group and of the real quadratic irrational numbers.
\newblock {\em Adv. Appl. Math. 130\/} (2021).

\bibitem{LMGCom}
{\sc Leclere, L., and Morier-Genoud, S.}
\newblock Quantum continuants, quantum rotundus and triangulations of annuli. 
\newblock {\em Electron. J. Combin. 30\/} (2023), Paper No. 3.35, 28 pp.

\bibitem{LMGOV}
{\sc Leclere, L., Morier-Genoud, S., Ovsienko, V., and Veselov, A.}
{\it On radius of convergence of $q$-deformed real numbers.}
{\em Mosc. Math. J. 24\/} (2024), 1--19.


\bibitem{McCSS}
{\sc McConville, T., Sagan, B.~E., and Smyth, C.}
\newblock On a rank-unimodality conjecture of {M}orier-{G}enoud and {O}vsienko.
\newblock {\em Discrete Math. 344\/}, 8 (2021).

\bibitem{MGOfmsigma}
{\sc Morier-Genoud, S., and Ovsienko, V.}
\newblock {$q$}-deformed rationals and {$q$}-continued fractions.
\newblock {\em Forum Math. Sigma 8\/} (2020), e13, 55 pp.

\bibitem{MGOexp}
{\sc Morier-Genoud, S., and Ovsienko, V.}
\newblock On $q$-deformed real numbers.
\newblock {\em Exp. Math. 31 (2022), 652--660.}

\bibitem{MGOV}
{\sc Morier-Genoud, S., Ovsienko, V., and Veselov, A.}
\newblock Burau representation of braid groups and $q$-rationals.
\newblock {\em Int. Math. Res. Not. IMRN} 2024, no. 10, 8618--8627.

\bibitem{OgRa}
{\sc Oguz, E.~K., and Ravichandran, M.}
\newblock Rank polynomials of fence posets are unimodal.
\newblock {\em Discrete Math. 346}, 2 (2023), Paper No. 113218.

\bibitem{Oven}
{\sc Ovenhouse, N.}
\newblock $q$-rationals and finite Schubert varieties.
\newblock C. R. Math. Acad. Sci. Paris 361 (2023), 807--818.

\bibitem{OP}
{\sc Ovsienko, V., and Pedon, E.}
\newblock Continued fractions for $q$-deformed real numbers, 
$\{-1,0,1\}$-Hankel determinants, 
and Somos-Gale-Robinson sequences.
\newblock ArXiv:2312.17009.
 
\bibitem{RenB}
{\sc Ren, X.}
\newblock  On radiuses of convergence of $q$-metallic numbers and related $q$-rational numbers. 
\newblock {\em Res. Number Theory 8} (2022).

\bibitem{Sik}
{\sc Sikora, A.}
\newblock  Tangle equations, the Jones conjecture, slopes of surfaces in tangle complements, and q-deformed rationals. 
\newblock {\em Canad. J. Math. 76} (2024), 707--727.


\end{thebibliography}
\end{document}